\documentclass[a4paper]{article}
\usepackage{latexsym,amssymb,amsmath,amsthm}
\newcommand{\vargra}[2]{\calg_{#1,#2}}
\newcommand{\immplu}[2]{{\wp_{#1,#2}}}
\newcommand{\posu}{{\uparrow}}
\newcommand{\pogiu}{{\downarrow}}
\newcommand{\calb}{\mathcal{B}}
\newcommand{\calf}{\mathcal{F}}
\newcommand{\calg}{\mathcal{G}}
\newcommand{\calp}{\mathcal{P}}
\newcommand{\cals}{\mathcal{S}}
\newcommand{\calt}{\mathcal{T}}
\newcommand{\PG}[2]{\mathrm{PG}(#1,#2)}
\newtheorem{prop}{Proposition}
\newtheorem{teor}[prop]{Theorem}
\newtheorem{cor}[prop]{Corollary}
\theoremstyle{remark}{\newtheorem{rem}{Remark}}

\begin{document}

\sloppy

\title{Incidence and Combinatorial Properties of Linear Complexes}

\author{Hans Havlicek\\
Institut f\"{u}r Diskrete Mathematik und Geometrie\\ Technische Universit\"{a}t Wien\\
Wiedner Hauptstra{\ss}e 8--10\\ A-1040 Wien\\ Austria\\
\texttt{havlicek@geometrie.tuwien.ac.at}
\and Corrado Zanella\\ Dipartimento di Tecnica e Gestione dei Sistemi
Industriali\\ Universit\`{a} di Padova\\ Stradella S.\ Nicola 3\\ I-36100 Vicenza\\
Italy\\
\texttt{corrado.zanella@unipd.it}}

\date{}

\maketitle

\centerline{\it Dedicated to Helmut Karzel on the occasion of his 80th
birthday}

\begin{abstract}
In this paper a generalisation of the notion of polarity is exhibited which
allows to completely describe, in an incidence-geometric way, the linear
complexes of $h$-subspaces. A generalised polarity is defined to be a partial
map which maps $(h-1)$-subspaces to hyperplanes, satisfying suitable linearity
and reciprocity properties. Generalised polarities with the null property give
rise to a linear complexes and vice versa. Given that there exists for $h>1$ a
linear complex of $h$-subspaces which contains no star --this seems to be an
open problem over an arbitrary ground field --the combinatorial structure of a
partition of the line set of the projective space into non-geometric spreads of
its hyperplanes can be obtained. This line partition has an additional
linearity property which turns out to be characteristic.

MSC 2000: 51A45, 51E20, 51E23, 05B25, 05B40

Keywords: Linear complex, Grassmannian, linear mapping, polarity, line spread,
line partition.
\end{abstract}

\section{Introduction}

The notion of a linear complex has been investigated for longer than a century,
and one could regard it as completely known. The classical approach is to
define a linear complex of $h$-dimensional subspaces (shortly: $h$-subspaces)
in $\PG nF$, i.e., the $n$-dimensional projective space over a commutative
field $F$, as the set of all $h$-subspaces whose Grassmann coordinates satisfy
a non-trivial linear equation or, in a coordinate-free way, via a hyperplane
section of the corresponding Grassmann variety. It is well known that each
linear complex of \emph{lines\/} ($h=1$) is the set of all lines which are
contained in their polar subspace with respect to a (possibly degenerate) null
polarity. The classification of these complexes reduces to the well known
classification of non-zero alternating matrices over $F$. Also, a lot was
written by leading classical authors on the theory of linear complexes of
$h$-subspaces for $h>1$, but seemingly, such a theory was developed for the
complex numbers only. As regards the topics which will be dealt with in this
paper see, among others, \cite{Ba52}, \cite{Bo53}, \cite{Bu61}, \cite{Co34},
\cite {Mo34}, and \cite{Se18}.
\par
To our knowledge, no incidence-geometric approach to linear complexes exists
for $h>1$. Our goal is to adopt such a point of view and coherently describe
linear complexes of $h$-subspaces. To accomplish this task, we first collect
some results about \emph{linear mappings\/} and \emph{primes\/} (geometric
hyperplanes) of Grassmannians over arbitrary fields (commutative or skew). In a
Grassmannian over a proper skew field there is only one kind of prime, whence
we rule out skew fields at an early stage of our investigation. In a
Grassmannian over a commutative field primes and linear complexes are the same,
but this result gives no explicit information about linear complexes.
\par
In Section~\ref{se:polaritaet} the notion of \emph{polarity\/} is generalised.
This leads us to an incidence-geometric construction of linear complexes of
$h$-subspaces in $\PG nF$ in terms of \emph{generalised null polarities}. A
crucial problem for linear complexes of $h$-subspaces in $\PG nF$ is the
existence of \emph{singular $(h-1)$-subspaces} for $h>1$. This problem is
addressed in Section~\ref{se:singulaer}, where we also sketch a result from
\cite{Ba52} for the case of complex numbers. However, a solution over an
arbitrary field $F$ presently seems beyond reach. If there exists a linear
complex of $h$-subspaces, $h>1$, without singular $(h-1)$-subspaces, then there
is also a linear complex of planes without singular lines. Therefore, we focus
on linear complexes of planes without singular lines, where we establish an
intriguing relation with \emph{linear line spreads\/}. Any linear complex of
planes without singular lines gives rise to linear line spread, which turns out
to be non-geometric in several cases, e.~g. for a finite ground field. Finally,
Section~\ref{se:partition} is devoted to a connection between linear complexes
of planes without singular lines and \emph{linear line partitions}. This
connection is used to show that certain projective spaces do not admit linear
line partitions.

\section{The Grassmannians of a projective space}

A \emph{semilinear space\/} is a pair $\Sigma=(\calp,\calb)$, where $\calp$ is
a set of \emph{points}, $\calb\subseteq2^{\calp}$ is a set of \emph{lines}, and
the following axioms hold:
\begin{itemize}
\item[(\emph{i\/})]
$|\ell|\ge2$ for each $\ell\in\calb$.

\item[(\emph{ii\/})]
For every $X\in\calp$ there exists at least one $\ell\in\calb$ such that
$X\in\ell$.

\item[(\emph{iii\/})]
$|\ell\cap\ell'|\le1$ for every $\ell,\ell'\in\calb$, $\ell\neq\ell'$.
\end{itemize}
Note that some authors use slightly different axioms for a semilinear space;
others speak of a \emph{partial linear space\/} instead.
\par
Given two points $X,Y\in\calp$ we write $X\sim Y$, where ``$\sim$'' is to be
read as \emph{collinear}, if there is an $\ell\in\calb$ such that $X,Y\in\ell$.
Otherwise, $X$ and $Y$ are said to be \emph{non-collinear}, $X\not\sim Y$. If
$X\sim Y$ and $X\neq Y$, then the unique line $\ell\in\calb$ such that
$X,Y\in\ell$ is denoted by $X Y$.
\par
Let $F$ be a (commutative or non-commutative) field. We denote by $\PG nF$ the
$n$-dimensional projective space coordinatised by the field $F$ and collect
some basic notions which will be used throughout this article:
\par
If $X$ is a subspace of $\PG nF$ and $\dim X=d$, then $X$ is called a
\emph{$d$-subspace}. Let $U$ and $W$ be two subspaces such that $U\subseteq W$.
The \emph{interval\/} $[U,W]$ is the set of all subspaces of $\PG nF$
containing $U$ and contained in $W$. The set of all $d$-subspaces of $\PG nF$
belonging to $[U,W]$ is denoted by $[U,W]_d$. In particular, for an interval
$[U,W]$ with $\dim U=h-1$ and $\dim W= h+1$ the set $[U,W]_h$ is the
\emph{pencil\/} of $h$-subspaces determined by $U$ and $W$.
\par
Let $[U,W]$ be an interval in $\PG nF$, $\dim U=:k-1$. Then $[U,W]$ carries the
structure of a projective space, whose $d$-dimensional subspaces are precisely
the $(k+d)$-subspaces of $\PG nF$ containing $U$ and contained in $W$. The
projective space $[U,W]$ is isomorphic to $\PG{\dim W-k}F$.
\par
The \emph{$h$-th Grassmannian\/} of $\PG nF$, $0\le h\le n-1$, is the
semilinear space $\Gamma(n,h,F)=(\calp,\calb)$, where $\calp$ is the set of all
$h$-subspaces of $\PG nF$, and $\calb$ is the set of all pencils of
$h$-subspaces. When the field is clear from the context we simply write
$\Gamma(n,h)$. If $\PG nF$ admits a correlation (e.~g., for a commutative field
$F$) then $\Gamma(n,h,F)$ is isomorphic to $\Gamma(n,n-h-1,F)$. In order to
avoid confusion, the elements of $\calp$ and $\calb$ will be called
\emph{G-points\/} and \emph{G-lines}, respectively, where ``G'' abbreviates
``Grassmann''.
\par
We consider the Grassmannian $\Gamma(n,h,F)$ for a fixed $h$. If a set $I$ of
G-points is a $d$-dimensional projective space with respect to the G-lines
contained in $I$, then $I$ is called a \emph{$d$-G-subspace}. A $d$-G-subspace
$I$ is thus a set of $h$-subspaces of $\PG nF$ pairwise meeting in
$(h-1)$-subspaces. Hence there exists an $(h-1)$-subspace $U$ of $\PG nF$ such
that $U\subset X$ for all $X\in I$, or there exists an $(h+1)$-subspace $V$ of
$\PG nF$ such that $X\subset V$ for all $X\in I$; both conditions hold for
$d\le1$.
\par
Next, we describe the maximal G-subspaces of $\Gamma(n,h,F)$. We start by
recalling two notions: If $U$ is an $(h-1)$-subspace of $\PG nF$, the
\emph{star\/} with \emph{centre\/} $U$ is the set of all $h$-subspaces of $\PG
nF$ containing $U$. Consequently, a star is an $(n-h)$-G-subspace. Let $\cals$
denote the set of stars. A \emph{dual star\/} is the set of all $h$-subspaces
of $\PG nF$ contained in a fixed $(h+1)$-subspace. The dual stars are
$(h+1)$-G-subspaces, and the set of all dual stars will be denoted by $\calt$.
For $h=1$ a dual star is a \emph{ruled plane}. If $h=0,n-1$, then $\calp$ is
the only maximal G-subspace. Otherwise, the set of all maximal G-subspaces
partitions into the families $\cals$ and $\calt$.

\section{Linear mappings}

Let $\calp$ and $\calp'$ be sets. We say that $\chi$ is a \emph{partial map\/}
of $\calp$ into $\calp'$, if there is a subset $\mathbb D(\chi)$ of $\calp$
such that $\chi: \mathbb D(\chi)\rightarrow\calp'$ is a map in the usual sense.
The set $\mathbb D(\chi)$ is called the \emph{domain\/} of $\chi$. Let $\mathbb
A(\chi)=\calp\setminus\mathbb D(\chi)$ be the \emph{exceptional subset\/} of
$\chi$. The elements of $\mathbb A(\chi)$ are called \emph{exceptional
points\/} of $\chi$. If $\mathbb A(\chi)=\emptyset$, then $\chi$ is
\emph{global}. We maintain the notation $\chi: \calp\rightarrow\calp'$ even
when $\chi$ is not global.
\par
In order to have a less complicated notation we extend the definition of $\chi$
to the power set of $\calp$ by defining
\begin{equation*}
    \phi^{\chi}:=(\phi\cap\mathbb D(\chi))^{\chi}
    = \{X^\chi\mid X\in \phi\cap\mathbb D(\chi)\}
    \mbox{~for every~} \phi\in 2^\calp.
\end{equation*}
Thus $\chi$ assigns to \emph{every\/} subset of $\calp$ a subset of $\calp'$.
\par
Next, let $\Sigma=(\calp,\calb)$ and $\Sigma'=(\calp',\calb')$ be semilinear
spaces. A partial map $\chi: \calp\rightarrow\calp'$ is called a \emph{linear
mapping\/}, if for each $\ell\in\calb$ one of the following holds:
\begin{itemize}
\item [(\emph{i\/})] $\ell^{\chi}\in\calb'$, and $\chi$ maps $\ell$ bijectively onto
$\ell^{\chi}$.
\item[(\emph{ii\/})] $\ell^{\chi}=\{P'\}$, where $P'\in\calp'$, and
$|\ell\cap\mathbb A(\chi)|=1$.
\item[(\emph{iii\/})] $\ell\subseteq\mathbb A(\chi)$.
\end{itemize}
Obviously, these three conditions are mutually exclusive. This linear mapping
will also be denoted by $\chi: \Sigma\rightarrow\Sigma'$.
\par
If $X, Y$ are distinct collinear points in $\mathbb A(\chi)$, then the line
$XY$ is a subset of $\mathbb A(\chi)$ by (iii). In particular, for a projective
space $\Sigma$ this implies that $\mathbb A(\chi)$ is a subspace of $\Sigma$.
In the case that $\chi$ is global, condition (i) holds for all lines, whence
distinct collinear points have distinct images. However, the images of
non-collinear points may coincide.
\par
There is one type of linear mapping deserving special mention: A \emph{full
projective embedding\/} is a global and injective linear map
$\chi:\Sigma\to\Sigma'$ for which $\Sigma'$ is a projective space.
\par
The linear mappings between projective spaces allow the following explicit
description:

\begin{teor}{\rm\cite{Br73,Ha81}}\label{t:brauner}
Let $\Sigma$ and $\Sigma'$ be projective spaces, and let $\chi:
\Sigma\rightarrow\Sigma'$ be a linear mapping. Then the partial map $\chi$
splits into a projection from $\mathbb A(\chi)$ onto a complementary subspace
in $\Sigma$, say $U$, and a collineation between $U$ and a subspace of
$\Sigma'$.
\end{teor}
The linear mappings between Desarguesian projective spaces are --up to one
particular case -- the geometric counterpart of the semilinear maps between
vector spaces; the exceptional points correspond to non-zero vectors in the
kernel. The situation is different if the image of the linear mapping is a
line. Here the collineation from the above theorem is just a bijection which,
of course, need not be induced by a semilinear bijection between the underlying
vector spaces.
\par
We now shortly exhibit full projective embeddings of a Grassmannian
$\Gamma(n,h,F)$. In doing so, the Grassmannians $\Gamma(n,0,F)$ and
$\Gamma(n,n-1,F)$ will be excluded, since they are a priori projective spaces.
Hence we have $0<h<n-1$ which implies $n\ge 3$.
\par
First, assume that $F$ is a non-commutative field. By \cite[Satz~3.6]{Ha81}, no
full projective embedding of $\Gamma(n,h,F)$ exists for $0<h<n-1$.
\par
Next, we assume $F$ to be a commutative field. Let $N={n+1 \choose h+1}-1$ and
\begin{equation*}
    \immplu nh: \Gamma(n,h,F)\rightarrow\PG NF
\end{equation*}
be the map defined by $X^{\immplu nh}=F(v_0\land v_1\land\ldots\land v_h)$,
where $v_0, v_1, \ldots, v_h$ is a basis of $X$ meant as an $(h+1)$-dimensional
subspace of $F^{n+1}$, and $\PG NF$ is represented as $(h+1)$-th exterior power
of $F^{n+1}$. This definition is independent on the choice of the basis, and
$\immplu nh$ is a full projective embedding, called the \emph{Pl\"{u}cker
embedding}. Its image $\calp^{\wp_{n,h}}=\vargra nh$ is a \emph{Grassmann
variety}. Each Grassmann variety is intersection of quadrics
\cite[pp.~184--188]{Se61}; in particular, $\vargra 31$ is the well-known
\emph{Klein quadric\/}.
\par
So, the question remains of describing all full projective embeddings of
$\Gamma(n,h,F)$ when $F$ is commutative. An answer can be given in terms of
linear mappings.

\begin{teor}{\rm\cite{Ha81}}\label{t:havlicek}
Let $F$ be a commutative field and let $\Sigma'$ be a projective space. If
$\chi: \Gamma(n,h,F)\rightarrow\Sigma'$ is a linear mapping, then there exists
a unique linear mapping $\mu: \PG NF\rightarrow\Sigma'$ such that $\chi=\immplu
nh\mu$.
\end{teor}
By this universal property, we may obtain all full projective embeddings of
$\Gamma(n,h,F)$ (to within collineations) as a product of $\immplu nh$ by a
(possibly trivial) projection whose centre does not meet any secant of the
Grassmann variety $\vargra nh$. Therefore, each fully embedded Grassmannian is
a (possibly trivial) projection of a Grassmann variety $\vargra nh$. In
\cite{We83} and \cite{Za95} sufficient conditions for the
\mbox{(non-)}\hspace{0pt}
existence of a non-trivial projection are given.
\par
For other questions and literature on linear mappings the reader is referred to
the book \cite{Fa00} (different terminology), \cite{Br73} (including a survey
of older literature), and \cite{Ha94}.

\section{Primes and linear complexes}\label{se:complexes}

A \emph{prime\/} of a semilinear space $(\calp,\calb)$ is a proper subset $L$
of $\calp$, such that for each $\phi\in\calb$ either $\phi\subseteq L$, or
$|\phi\cap L|=1$. Note that some authors call such a subset a \emph{geometric
hyperplane}.
\par
From now on $\Gamma(n,h,F)=(\calp,\calb)$ is a Grassmannian. It is easily seen
that the set of all $h$-subspaces having non-empty intersection with a fixed
$(n-h-1)$-subspace is a prime of $\Gamma(n,h,F)$. Somewhat surprisingly,
$\Gamma(n,h,F)$ has no other primes, if $F$ is a non-commutative field. This
beautiful result can be found in \cite{HaSh93}.
\par
Next, let $F$ be a commutative field and $\immplu nh:
\Gamma(n,h,F)\rightarrow\PG NF$ the related Pl\"{u}cker embedding. If $H$ is a
hyperplane of $\PG NF$, then
\begin{equation*}
    K=(H\cap\vargra nh)^{\immplu nh^{-1}}
\end{equation*}
is called a \emph{linear complex of $h$-subspaces\/} in $\PG nF$. Cf., e.~g.\
\cite[p.~322]{Bu61}. If $\varphi$ is a G-line, then either $\varphi\subseteq
K$, or $|\varphi\cap K|=1$. Since $\vargra nh$ generates $\PG NF$, we have
$K\neq\calp$. Thus each linear complex is a prime of $\Gamma(n,h,F)$.
Conversely, we have the following result; its short proof is taken from
\cite[p.~179]{Ha81}:
\begin{prop}\label{p:primi}
If $F$ is commutative, then each prime of\/ $\Gamma(n,h,F)$ is a linear
complex.
\end{prop}

\begin{proof}
Let $L$ be a prime of $\Gamma(n,h,F)$. Define a partial map $\chi:
\Gamma(n,h,F)\rightarrow\Sigma'$, where $\Sigma'$ is a point, by setting
$\mathbb A(\chi)=L$. Obviously, this $\chi$ is linear. By
Theorem~\ref{t:havlicek}, $\chi=\immplu nh\mu$, where $\mu: \PG
NF\rightarrow\Sigma'$ is a linear mapping. Then
\begin{equation*}
    X\in L\ \Leftrightarrow\ X\in\mathbb A(\chi)\
    \Leftrightarrow\
    X^{\immplu nh}\in\mathbb A(\mu),
\end{equation*}
and the latter exceptional set is a hyperplane by
Theorem~\ref{t:brauner}.
\end{proof}

In \cite{Sh92} a self-contained proof of Proposition~\ref{p:primi} is given.

\section{Generalised polarities arising from a linear complex}\label{se:polaritaet}

Up to the end of the paper $F$ denotes a \emph{commutative\/} field.
\par
Our first aim is to generalise the concept of polarity. Suppose that
$\chi:\Gamma(n,k) \rightarrow \PG nF^*$, $0\le k\le n-1$, is linear. So, $\chi$
is a partial map of the set of $k$-subspaces of $\PG nF$ into its hyperplane
set. We say that such a $\chi$ is a (\emph{generalised\/}) \emph{polarity\/} if
for all $U_1,U_2\in\Gamma(n,k)$ with $U_1\sim U_2$ the following holds:
\begin{equation}\label{e:polarita}
    U_1\subseteq U_2^\chi \mbox{ implies } U_2\subseteq U_1^\chi.
\end{equation}
In (\ref{e:polarita}) it is understood that for $U\in\mathbb A(\chi)$,
$U^\chi=\{U\}^\chi$ is the whole projective space. When speaking of
``polarities'' below, we always mean ``generalised polarities''.
\par
For $k=0$ our definition is in accordance with the usual definition of a
(possibly degenerate) polarity. Likewise, the following result is well known
for $k=0$ in the non-degenerate case. See, e.~g.\ \cite[p.~110,
Cor.~1]{Baer52}: \emph{``Null systems are polarities''}.

\begin{teor}\label{t:nullpol}
Let $\chi: \Gamma(n,k)\rightarrow\PG nF^*$, $0\le k\le n-1$, be a linear
mapping, and assume that $\chi$ satisfies the\/ \emph{null property}
\begin{equation}\label{e:nullprop}
    U\subseteq U^\chi \mbox{~~for each~~}U\in\mathbb D(\chi).
\end{equation}
Then $\chi$ is a polarity.
\end{teor}

\begin{proof}
Assume that $U_1$ and $U_2$ are distinct $k$-subspaces such that $U_1\sim U_2$
in $\Gamma(n,k)$. By the linearity of $\chi$, one of the following holds:
(\emph{i\/})~$U_1,U_2\in\mathbb D(\chi)$ and $U_1^\chi\neq U_2^\chi$; then
$(U_1 U_2)^\chi$ is the pencil of hyperplanes through $U_1^\chi\cap U_2^\chi$.
(\emph{ii\/})~There is a unique $U_0\in U_1 U_2$ such that $U_0\in\mathbb
A(\chi)$, and a hyperplane $E$ of $\PG nF$ exists such that $U\in U_1 U_2$,
$U\neq U_0$ implies $U^\chi=E$. (\emph{iii\/})~$U_1,U_2\in\mathbb A(\chi)$ so
that $U_1 U_2\subseteq\mathbb A(\chi)$.
\par
Let $P$ be a point. In any case we have
\begin{equation}\label{e:(a)}
    P\in U_1^\chi\cap U_2^\chi \Rightarrow (P\in U^\chi \mbox{~~for all~~}U\in U_1
    U_2).
\end{equation}
In order to prove (\ref{e:polarita}), assume that the above subspaces satisfy
$U_1\subseteq U_2^\chi$. Since $U_1\subseteq U_1^\chi$, by (\ref{e:(a)}) we
have, for each $U\in U_1 U_2$, that $U_1\subseteq U^\chi$. On the other hand,
$U\subseteq U^\chi$ implies $(U\vee U_1)\subseteq U^\chi$, where $U\vee U_1$
denotes the join in $\PG nF$. Now take $U_1',U_2'\in U_1 U_2$ with $U_1'\neq
U_2'\neq U_1\neq U_1'$. For $i=1,2$ we obtain
\begin{equation*}
  U_1\vee U_2=U_i'\vee U_1\subseteq {U_i'}^\chi.
\end{equation*}
So, (\ref{e:(a)}) yields $U_1\vee U_2\subseteq U_1^\chi$, and in particular
$U_2\subseteq U_1^\chi$.
\end{proof}

Taking into account the above theorem, a linear mapping
$\Gamma(n,k)\rightarrow\PG nF^*$ with the null property (\ref{e:nullprop}) will
be called a (\emph{generalised}) \emph{null polarity}.
\par
Now we turn back to linear complexes:

\begin{prop}\label{p:derivato}
Let $K$ be a linear complex of $h$-subspaces of\/ $\Sigma=\PG nF$, $1\le h\le
n-1$, and let $[U,W]$ be an interval of\/ $\Sigma$ with $\dim U\le h-1$ and
$\dim W\ge h+1$. Then
\begin{equation*}
    K(U,W) := K \cap [U,W]_h,
\end{equation*}
i.e., the set consisting of all elements of $K$ containing $U$ and contained in
$W$, is a linear complex of $(h-1-\dim U)$-subspaces in the projective space
$[U,W]$, unless $[U,W]_h\subseteq K$.
\end{prop}

\begin{proof}
A pencil $\varphi$ of $(h-1-\dim U)$-subspaces with respect to the projective
space $[U,W]$ is a pencil of $h$-subspaces in $\Sigma$; from $\varphi\subseteq
K$ or $|\varphi\cap K|=1$ we have $\varphi\subseteq K(U,W)$ or $|\varphi\cap
K(U,W)|=1$, respectively. Therefore, if $[U,W]_h\not\subseteq K$, then $K(U,W)$
is a prime, and hence a linear complex, of the $(h-1-\dim U)$-th Grassmannian
of $[U,W]$.
\end{proof}
\par
As a particular case of the above proposition, if $P$ is a point of $\Sigma$,
then the set $K_P=K(P,\Sigma)$, consisting of all elements of $K$ incident with
$P$, is a linear complex of $(h-1)$-subspaces in the $(n-1)$-dimensional
projective space $[P,\Sigma]$, unless $K_P=[P,\Sigma]_h$.
\par
As a further consequence, if $S=[U,\Sigma]_h$ is the star with centre an
$(h-1)$-subspace $U$, then either (\emph{i\/})~$S\subseteq K$, or
(\emph{ii\/})~there is a hyperplane $E$ of $\PG nF$ such that for each $X\in
S$, we have $X\in K$ if, and only if, $X\subseteq E$. If (i) holds, then $U$ is
called a \emph{singular $(h-1)$-subspace\/} of $K$; otherwise $E$ is the
\emph{polar hyperplane\/} of $U$. Similarly, if $T=[\emptyset,V]_h$, $\dim
V=h+1$, is a dual star, then either (\emph{i\/})~$T\subseteq K$, or
(\emph{ii\/})~there is a point $P$ of $\PG nF$ such that for each $X\in T$, we
have $X\in K$ if and only if $P\in X$. If (i) holds, then $V$ is called a
\emph{total $(h+1)$-subspace\/} of $K$; otherwise $P$ is the \emph{pole\/} of
$V$.
\par
We are now in a position to introduce the following crucial notion. If $K$ is a
linear complex of $h$-subspaces, $1\le h\le n-1$, we will denote by ${\uparrow}
K$ the partial map of the set of all $(h-1)$-subspaces of $\PG nF$ into the
dual projective space $\PG nF^*$, defined as follows: given an
$(h-1)$-subspace, say $U$, if $U$ is singular, then $U\in\mathbb A(\posu K)$;
otherwise $U^{\posu K}$ is the polar hyperplane of $U$, i.e., the union of all
elements of $K$ containing $U$. In view of Theorem~\ref{t:upolarity} below,
this mapping $\posu K$ will be called the \emph{null polarity defined by $K$}.

\begin{teor}\label{t:upolarity}
Let $K$ be a linear complex of $h$-subspaces in $\Sigma=\PG nF$, $1\le h\le
n-1$. Then the following assertions hold:
\begin{itemize}
\item[\rm(a)]
The partial map $\posu K: \Gamma(n,h-1)\rightarrow\PG nF^*$ is a null polarity
with non-empty domain.
\item[\rm(b)]
The image of $\posu K$ generates a subspace of $\PG nF^*$ with dimension at
least $h$.
\end{itemize}
\end{teor}

\begin{proof}
(a) We use induction on $h$; for $h=1$ this is well known \cite[p.~322]{Bu61}.
So, let $h>1$ and let $\varphi$ be a pencil of $(h-1)$-subspaces. There is a
point, say $P$, incident with every $U\in\varphi$. If $K_P=[P,\Sigma]_h$, then
$\varphi\subseteq\mathbb A(\posu K)$. Otherwise, by induction assumption,
$\posu K_P$ is a linear mapping. Thus the linearity of $\posu K$ follows by
observing that the restrictions of $\posu K$ and $\posu K_P$ to $\varphi$
coincide. Since $U\subseteq U^{\posu K}$ holds for all $U\in\mathbb D(\posu K)$
by definition, $\posu K$ is a null polarity according to
Theorem~\ref{t:nullpol}. Finally, $K\neq\Gamma(n,h)$ implies that not all
$(h-1)$-subspaces can be singular, whence the domain of $\posu K$ is non-empty.
\par
(b) Assume to the contrary that the image of $\posu K$ generates a subspace of
$\PG nF^*$ with a smaller dimension. So the intersection of all hyperplanes in
the image of $\posu K$ contains an $(n-h)$-subspace $X$, say. We show that this
implies the contradiction $\mathbb D(\posu K)=\emptyset$: Given an
$(h-1)$-subspace $U$ we argue by induction on $k:=\dim U\cap X$. For $k=-1$
there is no hyperplane passing through $U\vee X$, whence $U\in\mathbb A(\posu
K)$. Next, assume $k>-1$. Then there is a point $P\in\Sigma$ outside $U\vee X$.
Also, there exists a point $Q\in U\cap X$. Let $V$ be a complement of $Q$ with
respect to $U$. Then $U$ is an element of the pencil $ [V,V\vee P\vee
Q]_{h-1}$. All other elements of this pencil meet $X$ in a subspace of
dimension $k-1$, whence they belong to $\mathbb A(\posu K)$ by the induction
hypothesis. The linearity of $\posu K$ yields $U\in\mathbb A(\posu K)$, as
required.\end{proof}

\begin{cor}\label{c:sing_h-1}
The set of all singular $(h-1)$-subspaces of the linear complex $K$ is equal to
$\left(\vargra n{h-1}\cap W\right)^{\wp_{n,h-1}^{-1}}$ for some subspace $W$
of\/ $\mathrm{PG}\!\left(\!{{n+1\choose h}-1},F\right)$ with
\begin{equation*}
    {n+1\choose h}-(n+2)\le\dim W\le{n+1\choose h}-(h+2).
\end{equation*}
\end{cor}

\begin{cor}\label{c:fasciosingolare}
If $U_1$ and $U_2$ are singular $(h-1)$-subspaces of $K$, and $U_1\sim U_2$ in
$\Gamma(n,h-1)$, then each element of the pencil determined by $U_1$ and $U_2$
is a singular $(h-1)$-subspace of $K$.
\end{cor}

We have seen that $\posu K$ is a null polarity. Conversely, we have:

\begin{teor}\label{t:converse}
Let $\chi: \Gamma(n,h-1)\rightarrow\PG nF^*$, $1\le h\le n-1$, be a linear
mapping with non-empty domain satisfying the null property
\emph{(\ref{e:nullprop})}. Then there is a unique linear complex of
$h$-subspaces, say $K$, such that $\chi=\posu K$.
\end{teor}

\begin{proof}
We define a set $K$ of $h$-subspaces by setting $X\in K$ if, and only if, there is
an $(h-1)$-subspace $U\subseteq X$ such that $X\subseteq U^\chi$. Any linear
complex with the required properties necessarily has to coincide with this $K$.
\par
The mapping $\chi$ is a polarity by Theorem~\ref{t:nullpol}. From
(\ref{e:polarita}) we infer that if $X\in K$, then $X\subseteq W^\chi$ for
every $(h-1)$-subspace $W$ of $X$. If $\varphi$ is the pencil of $h$-subspaces
determined by the $(h-1)$-subspace $U_\varphi$ and the $(h+1)$-subspace
$V_\varphi$, $U_\varphi\subseteq V_\varphi$, then the G-points of $\varphi$
belonging to $K$ are exactly the elements $X\in\varphi$ such that $X\subseteq
U_\varphi^\chi$. Therefore, either $\varphi\subseteq K$ or $|\varphi\cap K|=1$,
whence $K$ is a prime. The proof is now accomplished by applying
Proposition~\ref{p:primi}.
\end{proof}

The above theorem is a generalisation of the classical one, characterising a
linear complex of lines as a set arising from a (possibly degenerate) null
polarity $\PG nF \to {\PG nF}^*$.
\par
By dual arguments we obtain a linear mapping $\pogiu K:
\Gamma(n,h+1)\rightarrow\PG nF$, $0\le h\le n-2$, the \emph{dual polarity\/} of
$K$. There holds:
\begin{teor}
Let $\chi: \Gamma(n,h+1)\rightarrow\PG nF$, $0\le h\le n-2$, be a linear
mapping with non-empty domain, and assume that for each $V\in\mathbb D(\chi)$,
$V^\chi\subseteq V$. Then there is a unique linear complex of $h$-subspaces,
say $K$, such that $\chi=\pogiu K$.
\end{teor}

\section{Existence of singular $(h-1)$-subspaces}\label{se:singulaer}

We start with the following technical result.

\begin{prop}\label{p:k_proiez}
Assume that $k$, $h$, and $n$ are integers such that $1\le k\le h\le n-1$. Let
$K$ be a linear complex of $h$-subspaces of $\PG nF$ having no singular
$(h-1)$-subspace. Then $\PG{n+k-h}F$ contains a linear complex of $k$-subspaces
having no singular $(k-1)$-subspace.
\end{prop}

\begin{proof}
Let $W$ be any $(h-k-1)$-subspace of $\Sigma=\PG nF$. The intersection
$K_W=K\cap[W,\Sigma]_h$ is a set of $k$-subspaces of
$[W,\Sigma]\cong\PG{n+k-h}F$. If $L$ is a $(k-1)$-subspace of $[W,\Sigma]$,
then $L$ is an $(h-1)$-subspace of $\Sigma$, containing $W$. By assumption an
$h$-subspace $X$ of $\Sigma$ exists that contains $L$ and does not belong to
$K$. Such $X$ is a $k$-subspace of $[W,\Sigma]$, containing $L$ and not
belonging to $K_W$. Now Proposition~\ref{p:primi} easily yields the assertion.
\end{proof}

A linear complex $K$ of lines is the set of all self-conjugate lines of a null
polarity with non-empty domain in $\PG nF$. Since each alternating matrix has
even rank, $K$ always has a singular point for $n$ even; on the other hand, for
each odd $n$ there are linear complexes of lines without singular points. Thus,
by Proposition~\ref{p:k_proiez} with $k=1$, we have that if $n-h\equiv1\ \pmod
2$, then each linear complex of $h$-subspaces in $\PG nF$ has a singular
$(h-1)$-subspace. In our proof of Proposition~\ref{p:k_proiez} the subspace $W$
was chosen arbitrarily. Hence this result can be refined as follows:

\begin{prop}\label{p:npiuh}
Let $K$ be a linear complex of $h$-subspaces in $\PG nF$. If $1\le h\le n-1$
and $n-h\equiv 1 \pmod 2$, then each $(h-2)$-subspace in $\PG nF$ is contained
in a singular $(h-1)$-subspace.
\end{prop}

The question concerning the existence of total $(h+1)$-subspaces is dealt with
in \cite{Ba52} \emph{over the field $\mathbb C$ of complex numbers\/} as
follows. The \emph{product\/} of a linear complex of $h$-subspaces, say $K$,
and a linear complex of points, i.e., a hyperplane $H$, is the set
\begin{equation*}
    K\cdot H:=\{X\in\Gamma(n,h+1,\mathbb C)\mid \exists\, Y\in K : Y\subseteq
    X\cap H\}
\end{equation*}
If $[\emptyset,H]_h\subseteq K$ then $H$ is called a \emph{total hyperplane\/}
of $K$; in this case $K\cdot H$ is the set of all $(h+1)$-subspaces of
PG$(n,\mathbb C)$. Otherwise, $K\cdot H$ is a linear complex of
$(h+1)$-subspaces in PG$(n,\mathbb C)$. Let $H_0$, $H_1$, $\ldots$, $H_n$ be
independent hyperplanes in PG$(n,\mathbb C)$. Then the set $\Theta$ of all
total $(h+1)$-subspaces of $K$ is the intersection of $K\cdot H_0, K\cdot
H_1,\ldots, K\cdot H_n$; that is, $\Theta$ is represented on the Grassmann
variety ${\calg}_{n,h+1}$ as intersection with at most $n+1$ hyperplanes. (This
is dual to our Corollary~\ref{c:sing_h-1}, and holds over an arbitrary ground
field.) Since ${\calg}_{n,h+1}$ is an algebraic variety with dimension
$(h+2)(n-h-1)$, a sufficient condition for the existence of total
$(h+1)$-subspaces is
\begin{equation*}
    (h+2)(n-h-1)-(n+1)\ge0,
\end{equation*}
that is
\begin{equation*}
    (h+1)(n-h-2)\ge1.
\end{equation*}
Since $h+1\ge1$, the condition reads $n-h-2>0$ or, equivalently, $h<n-2$.
Summarising, we obtain:
\begin{teor}\cite{Ba52}\label{t:Ba52}
Each linear complex of $h$-subspaces in $\PG n{\mathbb C}$, with $h<n-2$, has a
total $(h+1)$-subspace.
\end{teor}

Dually, there holds:

\begin{cor}\label{c:Ba52}
Each linear complex of $h$-subspaces in $\PG n{\mathbb C}$, with $h\ge2$, has a
singular $(h-1)$-subspace.
\end{cor}

A \emph{line spread\/} of a projective space $\PG nF$ is a set of lines, say
$\calf$, such that each point of $\PG nF$ belongs to exactly one line of
$\calf$. A line spread is called \emph{geometric\/} (or \emph{normal\/}) if for
every pair of distinct lines of $\calf$, say $\ell$, $m$, the lines of $\calf$
in the solid $\ell\vee m$ form a spread of $\ell\vee m$. The line spread
$\calf$ is \emph{linear\/} if $\calf^{\immplu n1}$ is the intersection of the
Grassmann variety $\vargra n1$ with a subspace of its ambient space $\PG MF$
with $M:=(n^2+n-2)/2$.

\begin{prop}\label{p:spread}
Let $K$ be a linear complex of planes in $\PG nF$ having no singular line. Then
$n$ is even. For each hyperplane $H$ in $\PG nF$, let $\calf_H$ be the set of
all lines whose polar hyperplane is $H$. Then $\calf_H$ is a line spread of
$H$.
\end{prop}

\begin{proof}
By Proposition~\ref{p:npiuh}, $n$ is even. Let $A$ and $H$ be a point and a
hyperplane in $\PG nF$, respectively, such that $A\in H$. The set $S$ of all
lines of $\PG nF$ containing the point $A$ is a subspace of $\Gamma(n,1)$
isomorphic to $\PG{n-1}F$. Since $\posu K$ is global, its restriction to $S$ is
a collineation. As a consequence, $S^{\posu K}$ is the set of \emph{all\/}
hyperplanes through $A$. So, a unique line $\ell$ through $A$ exists such that
$\ell^{\posu K}=H$. This proves that $\calf_H$ is a line spread of $H$.
\end{proof}

\begin{prop}\label{p:lin_spread}
Under the assumptions of Proposition\/~\emph{\ref{p:spread}} the line spread
$\calf_H$ is linear. More precisely,
\begin{equation*}
    \calf_H^{\immplu n1}=R\cap{\vargra n1},
\end{equation*}
where $R$ is an $(M-n)$-subspace of the ambient space $\PG MF$ of the Grassmann
variety $\vargra n1$. Furthermore, in each of the following cases the line
spread $\calf_H$ is not geometric.
\begin{itemize}
\item[\rm(a)]
The field $F$ is quadratically closed.

\item[\rm(b)]
The field $F$ is finite.
\end{itemize}
\end{prop}

\begin{proof}
By Theorems \ref{t:havlicek} and \ref{t:upolarity}, the null polarity $\posu K$
can be written as
\begin{equation*}
    \posu K = \immplu n1\pi\kappa,
\end{equation*}
where $\immplu n1$ is the Pl\"{u}cker embedding, $\pi$ is a projection of $\PG MF$
from a subspace $C$ with dimension $(M-n-1)$ onto a complementary $n$-subspace
$D$, and $\kappa: D\rightarrow\PG nF^*$ is a collineation. The non-existence of
singular lines yields
\begin{equation*}
    C\cap\vargra n1 = \emptyset.
\end{equation*}
So, $\calf_H^\immplu n1$ is the intersection of $\vargra n1$ with the subspace
$R:=C\vee H^{\kappa^{-1}}$ of $\PG MF$.
\par
Now suppose that (a) or (b) holds. We assume that $\calf_H$ is a geometric line
spread. By the above, there is a hyperplane, say $J$, of $\PG MF$, such that
$J\cap R=C$, whence
\begin{equation*}
    J\cap\calf_H^\immplu n1=\emptyset.
\end{equation*}
As $\calf_H$ is geometric, there is a solid $U$ of $\PG nF$ such that the lines
of $U$ belonging to $\calf_H$ form a line spread of $U$, say $\calf_1$.
Furthermore, $\calf_1^\immplu n1$ is the intersection of $R$ with a quadric
$Q_5^+$ (the Klein quadric representing the lines of $U$). By the table in
\cite[pp.~29--31]{Hi85} (which remains true, mutatis mutandis, also for an
infinite field), $\calf_1^\immplu n1$ has to be an elliptic quadric $Q_3^-$.
This is impossible if (a) holds. On the other hand, for a finite field $F$ we
get from $J\cap Q_3^- = \emptyset$ the contradiction that $Q_3^-$ would have an
exterior plane \cite[p.~17]{Hi85}.\end{proof}
\par
The proof from the above cannot be carried over to all infinite fields, since
an elliptic quadric may have an exterior plane.

\begin{rem}
Proposition~\ref{p:k_proiez} for $k=2$ together with the previous theorem
implies the following: Any linear complex of $h$-subspaces, $h>1$, having no
singular $(h-1)$-subspaces, yields a non-geometric linear line spread if $F$
satisfies one of the conditions (a) or (b). The authors do not know, whether
under these circumstances non-geometric linear line spreads exist or not. In
case of their non-existence, Theorem~\ref{t:Ba52} of Baldassarri and
Corollary~\ref{c:Ba52} would also hold for projective spaces over finite or
quadratically closed fields.
\end{rem}

\section{Linear line partitions}\label{se:partition}

A \emph{line partition\/} $\Omega$ of a projective space $\PG nF$ is a
partition of its line set into line spreads of hyperplanes such that each
hyperplane contains precisely one of these spreads. Each line partition of $\PG
nF$ induces a surjective map $\pi_\Omega$ of the line set onto the dual space
$\PG nF{}^*$ as follows: It assigns to each line $\ell$ the unique hyperplane
containing the equivalence class of $\ell$. Observe that this $\pi_\Omega$ is
globally defined on the line set.
\par
Line partitions of finite projective spaces were investigated in
\cite{FHVa83,FHVa84,FHJiVa86,FHMuTo06,To05}. Here $n$ necessarily has to be
even. In particular, we quote the following result:
\begin{teor}\emph{\cite{FHVa84,FHJiVa86}}\label{t:bluff}
Each finite projective space $\PG{2^i-2}{q}$ with $i\ge 2$ admits a line
partition.
\end{teor}

A line partition $\Omega$ will be called \emph{linear\/} if $\pi_\Omega$ from
the above is a linear mapping. The linear line partitions are closely related
with certain linear complexes.

\begin{teor}\label{t:polaritagen}
Let $\Omega$ be a line partition of $\PG nF$. Denote by $K$ the set of all
planes $\varepsilon$ such that a line $\ell$ exists which satisfies the
condition
\begin{equation*}
    \ell\subseteq\varepsilon\subseteq\ell^{\pi_\Omega}.
\end{equation*}
This $K$ is a linear complex of planes if, and only if, $\Omega$ is linear. In
this case, $K$ has no singular line.
\end{teor}

\begin{proof}
If $\Omega$ is linear, then $\chi:=\pi_\Omega$ satisfies the assumptions of
Theorem~\ref{t:converse}. This implies that $K$ is a linear complex of planes
without singular lines.
\par
Conversely, assume that the given set $K$ is a linear complex of planes. Let
$\ell$ be a non-singular line. Then $\ell^{\pi_\Omega}$ and $\ell^{\posu K}$
are hyperplanes which obviously are identical. Next, we show that singular
lines do not exist. Assume to the contrary that $\ell_1'$ is a singular line.
Since $\posu K$ has a non-empty domain, also a non-singular line $\ell_2'$
exists. There is a line which has a point in common with $\ell_1'$ and
$\ell_2'$, respectively. This line is either singular or non-singular. Hence
there exists a pencil of lines containing a singular line $\ell_1$ and a
non-singular line $\ell_2$, say. Since $\posu K$ is linear, there is another
non-singular line $\ell_3\neq \ell_2$ in this pencil. We infer from the above
and from the linearity of $\posu K$ that
\begin{equation*}
    \ell_2^{\pi_\Omega}=\ell_2^{\posu K}=\ell_3^{\posu K}=\ell_3^{\pi_\Omega}.
\end{equation*}
Thus the distinct incident lines $\ell_2$ and $\ell_3$ belong to the same line
spread, a contradiction. Therefore ${\pi_\Omega}={\posu K}$ which in turn
implies the linearity of $\Omega$.
\end{proof}

By Proposition~\ref{p:spread} and Theorem~\ref{t:polaritagen}, there is a
bijective correspondence between linear complexes of planes without singular
lines and linear line partitions. The linear complex $K$ defined in the theorem
will be called the linear complex of planes \emph{related\/} to the linear line
partition $\Omega$.

\begin{prop}
No projective space $\PG 4F$ admits linear line partitions.
\end{prop}

\begin{proof}
A linear complex of planes, say $K$, in $\PG 4F$ is dual to a linear
complex of lines $K^*$ in $\PG 4F^*$. The singular lines of $K$ are dual to
ruled planes of $K^*$. Let $S$ be the set of all singular points of $K^*$
(actually, hyperplanes of $\PG 4F$). If $S$ is a dual plane, then obviously
$K^*$ contains ruled planes. Otherwise $S$ is a point and there is a line
$\ell$ of $K^*$ such that $S\not\in\ell$. The plane $S\vee\ell$ is a ruled
plane in $K^*$. This proves that each linear complex of planes in $\PG 4F$ has
at least one singular line.
\end{proof}

Some finite four-dimensional projective spaces admit line partitions. In
particular this holds in $\PG 4q$ for $q=2,3$ \cite{FHVa84} and for $q=5,8,9$
\cite{To05}. By the above all these line partitions are necessarily non-linear.

\begin{prop}\label{p:pg6q}
No finite projective space $\PG 6q$ admits linear line partitions.
\end{prop}

\begin{proof}
Let $K$ be a linear complex of planes in $\PG 6q$, without singular lines.
Denote by $K_H$ the set of planes in $K$ which are contained in some hyperplane
$H$. This $K_H$ is a linear complex of planes in the five-dimensional finite
projective space $H$. Let $\calf_H$ be the line spread in $H$ given according
to Proposition~\ref{p:spread}. The elements of $\calf_H$ are precisely the
singular lines of $K_H$. In \cite[Theorem~14]{FeZa07} it is shown that (to
within projective transformations) a unique linear complex of planes in $\PG
5q$ exists such that its singular lines form a line spread; moreover, this
spread turns out to be geometric. This contradicts
Proposition~\ref{p:lin_spread}.
\end{proof}

\begin{rem}
It seems to be unknown whether or not linear line partitions do exist. In
particular we do not know, if the line partitions of PG$(2^i-2,q)$ by Fuji-Hara
and Vanstone (cf.\ Theorem~\ref{t:bluff}) are linear. As a matter of fact, to
our knowledge, the proof of their existence is only sketched in the literature.
\end{rem}

\begin{rem}
The authors conducted an extensive computer-based search for linear complexes
of planes without singular lines in $\PG 8q$, $q=2,3$, but no examples were
found. Thus a proof for the (non)-existence of linear complexes of
$h$-subspaces, $h\ge 2$, having no singular $(h-1)$-subspaces remains an
enticing open problem, even in the finite case.
\end{rem}

Acknowledgement. The second author acknowledges financial support from the
National Research Project ``Strutture geometriche, Combinatoria e loro
Applicazioni'' of the Italian \emph{Ministero dell'Universit\`{a} e della Ricerca},
PRIN 2005.

\end{document}